\documentclass[a4paper]{article}
\usepackage[utf8]{inputenc}
\usepackage{graphicx}
\graphicspath{{./images/}}
\usepackage{fullpage}
\usepackage[affil-it]{authblk}
\usepackage{longtable}
\usepackage{amsfonts,amssymb,amsmath,amscd,amsthm,amsbsy,tikz}

\usepackage{array}
\usepackage{multirow}
\usepackage{multicol}
\usepackage{booktabs}
\usepackage[font=small,labelfont=bf]{caption}
\usepackage{enumerate}
\usepackage{mathtools}
\usepackage{float}

\usepackage[normalem]{ulem} %

\usepackage{setspace}
\numberwithin{equation}{section}

\usepackage{hyperref}

\newcommand{\R}{\mathbb{R}}
\newcommand{\N}{\mathbb{N}}

\newcommand{\id}{\mbox{id}}

\newcommand{\e}{\varepsilon}

\newcommand{\cH}{\mathcal{H}}
\newcommand{\rD}{\mathrm{d}}
\newcommand{\fG}{\mathfrak{g}}
\newcommand{\fM}{\mathfrak{m}}
\newcommand{\fH}{\mathfrak{h}}
\newcommand{\fX}{\mathfrak{X}}
\newcommand{\sg}{\mathsf{g}}
\newcommand{\ip}[2]{\langle #1,#2\rangle}

\DeclareMathOperator{\spn}{span}
\DeclareMathOperator{\Ad}{Ad}

\DeclareMathOperator{\SO}{SO}

\DeclareMathOperator{\vol}{vol}

\newtheorem{thm}{Theorem}[section]
\newtheorem{cor}[thm]{Corollary}
\newtheorem{conj}[thm]{Conjecture}
\newtheorem{prop}[thm]{Proposition}
\newtheorem{lem}[thm]{Lemma}

\theoremstyle{definition}
\newtheorem{defin}[thm]{Definition}
\newtheorem*{ack}{Acknowledgements}

\theoremstyle{remark}
\newtheorem*{rmk}{Remark}

\title{Incompressible Euler fluids on compact cohomogeneity one manifolds}
\author[1]{Timothy Buttsworth}
\author[2]{Max Orchard}

\affil[1]{\small School of Mathematics and Statistics\\
The University of New South Wales\\
Kensington, Sydney\\ NSW 2052, Australia\\ 
\emph{E-mail address}: {\tt t.buttsworth@unsw.edu.au}}
\vspace{0.5cm}
\affil[2]{\small School of Mathematics and Physics\\
The University of Queensland\\
St Lucia, Brisbane\\
QLD 4072, Australia\\ 
\emph{E-mail address}: {\tt m.orchard@uq.edu.au}}

\date{}

\begin{document}

\maketitle

\abstract{
Let $(M,\mathsf{g})$ be a connected and compact Riemannian manifold admitting an isometric action by a compact Lie group $G$ whose principal orbits have codimension one. We show that any $G$-invariant, smooth, and divergence-free vector field $u_0$ on $(M,\mathsf{g})$ initiates a $G$-invariant time-varying velocity-pressure pair $(u,p)$ which has time interval $\mathbb{R}$, is smooth, and solves the incompressible Euler fluid equations.}

\section{Introduction}
Let $(M,\mathsf{g})$ be a compact Riemannian manifold with (possibly empty) boundary. The incompressible Euler fluid equations for velocity field $u:[0,T)\to \fX(M)$ and scalar pressure field $p:[0,T)\to C^{\infty}(M)$ are given by (cf.~\cite[Chapter 17]{TaylorPDE3})
\begin{align}\label{Euler}
\begin{split}
    0&=\frac{du}{dt}+\nabla_{u}u+\nabla p,\\
    0&=\nabla\cdot u,
    \end{split}
\end{align}
with the boundary condition 
\begin{align}\label{BCs0}
   (u(t))(x)\in T_{x}(\partial M)\ \ \text{for each} \ x\in\partial M \ \text{and} \ t\in [0,T).
\end{align}
Here, $\nabla\cdot u$ means the divergence of $u$ with respect to the Riemannian metric $\sg$, $\nabla p$ means the gradient of the scalar function $p$ (with respect to $\sg$), and $\nabla_{u}u$ is the Levi-Civita connection of $u$. 
Differential geometers are primarily interested in solutions $(u,p)$ to this equation because the vector field $u$ can be integrated to give a geodesic on the space of volume-preserving diffeomorphisms of $M$ equipped with the $L^2(M,\mathsf{g})$ metric (see, for instance, \cite{ArnoldKhesin,EbinMarsden70} for a thorough treatment of this topic).

A fundamental problem in the theory of non-linear PDEs is to determine whether a smooth solution of \eqref{Euler}--\eqref{BCs0} on the interval $[0,T)$ can always be extended to a smooth solution on $[0,\infty)$. Since the Euler fluid equations are time-reversible (with $t\mapsto -t$, $u\mapsto -u$ and $p\mapsto p$), this question of smooth extensions forward in time is equivalent to asking if a smooth solution can always be extended indefinitely both forwards \textit{and backwards} in time. 

An enormous amount of effort has been exerted in an attempt to answer this question in the special case that $M$ is an open submanifold of $\mathbb{R}^3$ equipped with the standard Euclidean metric. For example, Elgindi \cite{Elgindi21} has found that if one replaces the \textit{smoothness} requirement of $u$ with $C^{1,\alpha}$-regularity for some $\alpha>0$, there are  solutions in $\mathbb{R}^3$ which cannot be extended indefinitely. Chen and Hou \cite{Chen23} have produced a lengthy computationally-assisted proof that there is a smooth solution on the periodic cylinder $B_0(1)\times \mathbb{R}$ which cannot be extended indefinitely. Related work on singularity formation in fluid flow can be found in \cite{Chen21,Elgindi23,Deepmind25,Wangetal23}.

Tao has recently made several contributions to this problem on Riemannian manifolds without boundary in \cite{Tao17,Tao19}, where he stated the following conjecture:
\begin{conj}\label{Taoconjecture}
    There exists a closed (compact and without boundary) Riemannian manifold $(M,\mathsf{g})$ and a smooth solution $(u,p)$ to \eqref{Euler} that cannot be extended indefinitely forwards in time.
\end{conj}
The two aforementioned singularity results do not settle this conjecture because either the solution is not smooth, or the manifold has boundary.

The purpose of this paper is to show that Conjecture \ref{Taoconjecture} is false if one restricts the scope to Riemannian manifolds $(M,\mathsf{g})$ and solutions $(u,p)$ that are all invariant under the action of a Lie group $G$ whose principal orbits in $M$ have codimension one. We obtain such a strong global existence result in this case because the divergence-free condition imposes powerful restrictions on the vector field $u$ when combined with cohomogeneity one symmetries. 
\begin{thm}\label{maintheorem}
    Suppose the compact Lie group $G$ acts with cohomogeneity one and isometrically on the closed and connected Riemannian manifold $(M,\mathsf{g})$. For any smooth, divergence-free and $G$-invariant vector field $u_0\in \fX(M)$, there exists a smooth $u:\mathbb{R}\to \fX(M)$ and $p:\mathbb{R}\to C^{\infty}(M)$ satisfying \eqref{Euler} and the initial condition $u(0)=u_0$. 
\end{thm}

It is unclear if Theorem \ref{maintheorem} still holds if we relax the strength of the symmetry assumption. Indeed, the most popular symmetry assumption used in the study of Euler fluids in $\mathbb{R}^3$ is the \textit{axisymmetric} assumption, i.e., the fluid is invariant under the usual action of $\SO(2)$ which rotates $\mathbb{R}^3$ around one of the axes. In this situation, almost all orbits have codimension two. In the aforementioned work on singularity formulation, the authors consider flows without \textit{swirl} (i.e., flows that are everywhere perpendicular to the orbits).

This paper is organised as follows: In Section 2, we collect the necessary background on the incompressible Euler fluid equations on Riemannian manifolds, including well-posedness, a blow-up criterion, and the preservation of symmetries under the flow. Section 3 is devoted to the geometry of compact cohomogeneity one manifolds: we describe the structure of the orbit space and the slice theorem, the form of $G$-invariant metrics and vector fields, and the smoothness conditions that $G$-invariant vector fields must satisfy near nonprincipal orbits. In Section 4, we prove Theorem \ref{maintheorem} by deriving the Euler equations in the $G$-invariant setting and establishing uniform $C^1$ bounds on the solution to invoke the aforementioned blow-up criterion, treating the cases $M/G \simeq [0,1]$ (where there are nonprincipal orbits) and $M/G \simeq S^1$ (where all orbits are principal) separately. Finally, Section 5 extends these results to compact cohomogeneity one manifolds with boundary.

\begin{ack}
    The authors are grateful to Artem Pulemotov, Ramiro Lafuente and Patrick Donovan for comments on an earlier version of this paper. We also thank Wolfgang Ziller for insightful discussions regarding smoothness conditions for invariant vector fields on cohomogeneity one manifolds. 

    Both authors were supported financially by the Australian Government through the Australian Research Council grant DE220100919 and a Research Training Program Scholarship, respectively.
\end{ack}

\section{Geometry of the incompressible Euler fluid equations}\label{PrelimI}

In this section, we recall some important facts about the incompressible Euler fluid equations \eqref{Euler}--\eqref{BCs0} on Riemannian manifolds. In doing so, we see that the problem \eqref{Euler}--\eqref{BCs0} is well-posed, give a criterion for an initially smooth solution to blow up in finite time, and establish that the problem is compatible with the notion of symmetry.

\subsection{Well-posedness and blow-up criterion}
First, we observe that \eqref{Euler}--\eqref{BCs0} is indeed the natural geometric extension of the classical incompressible Euclidean Euler fluid equations to Riemannian manifolds (possibly with boundary). 
\begin{prop}\label{Eulergeometric}
    Let $\phi:M\to \tilde{M}$ be a time-independent isometry between two compact Riemannian manifolds $(M,\mathsf{g})$ and $(\tilde{M},\tilde{\mathsf{g}})$ (possibly with boundary).  
    If $(u,p)$ is a solution to \eqref{Euler}--\eqref{BCs0} on $(M,\mathsf{g})$, then $(\tilde{u},\tilde{p})$ is a solution on $(\tilde{M},\tilde{\mathsf{g}})$, where $(\tilde{u},\tilde{p})$ is defined according to $\tilde{u}(\phi(x))=\rD\phi_x (u(x))$ and $\tilde{p}(\phi(x))=p(x)$.   
\end{prop}
\begin{proof}
    The geometric operations $u\mapsto \nabla_u u$, $u\mapsto \nabla\cdot u$, and $p\mapsto \nabla p$ are invariant under isometries and the metric is independent of time, so both equations of \eqref{Euler} are preserved by applying $\rD\phi$. The boundary condition \eqref{BCs0} is also preserved since $\rD\phi_x(T_x \partial M)=T_{\phi(x)}\partial\tilde{M}$.
\end{proof}

Next, we note that the problem \eqref{Euler}--\eqref{BCs0} is well-posed and also give a necessary condition for a solution to \eqref{Euler}--\eqref{BCs0} to blow up in finite time. 
\begin{thm}\label{Eulerwellposed}
    Let $(M,\mathsf{g})$ be a compact and connected Riemannian manifold, possibly with boundary. For each smooth divergence free vector field $u_0$ on $M$ satisfying $u_0(x)\in T_{x}(\partial M)$ for each $x\in \partial M$, there exists $T>0$, and a unique smooth solution $p,u$ (up to addition of a time-varying constant for $p$) of \eqref{Euler}--\eqref{BCs0} on $[0,T)$, subject to the initial condition $u(0)=u_0$. Moreover, if $T^*<\infty$ is the largest such $T$, then $\limsup_{t\to T^* } \|u(t)\|_{C^1(M,\mathsf{g})}=\infty$. 
\end{thm}
\begin{proof}
    This follows from results in \cite[Chapter 17]{TaylorPDE3}. Indeed, Theorem 3.2 gives existence and uniqueness of $u$ in the Sobolev space $H^k$ for any $k>\frac{n}{2}+1$, where $n$ is the dimension of $M$. If $u_0$ is smooth, we can apply this result for all large $k$, and uniqueness implies that the resulting solution must be independent of $k$, and is therefore smooth. The preamble to Section 3 then explains how $p$ is uniquely determined up to a time-varying additive constant, and $\nabla p$ is uniquely determined. The blow-up criterion is stated as Proposition 3.3 in the same chapter. 
\end{proof}

\subsection{Preservation of symmetries}
In this paper, we are interested in analysing highly symmetric solutions to \eqref{Euler}--\eqref{BCs0}. It is useful to know that the unique solution will automatically be highly symmetric if it starts this way. First, we explain what we mean by ``symmetry''. 

\begin{defin}\label{Ginvariant}
    Let $G$ be a Lie group acting smoothly on a smooth manifold $M$, with associated diffeomorphisms $\varphi_g:M\to M$, $x\mapsto g\cdot x$, for each $g\in G$. We say that
    \begin{itemize}
        \item a Riemannian metric $\mathsf{g}$ on $M$ is \emph{$G$-invariant} if $\varphi_g$ is an isometry of $(M,\mathsf{g})$ for each $g\in G$,
        \item a vector field $u$ on $M$ is \emph{$G$-invariant} if $u(\varphi_g(x))=d\varphi_g(u(x))\in T_{\varphi_g(x)}M$ for each $x\in M$ and $g\in G$, and 
        \item a smooth function $p:M\to \mathbb{R}$ is \emph{$G$-invariant} if $p(\varphi_g(x))=p(x)$ for all $x\in M$ and $g\in G$.
    \end{itemize}
    The set of all $G$-invariant vector fields on $M$ is denoted $\fX^G(M)$.
\end{defin}
We now show that the initial-value problem for the Euler fluid equations \eqref{Euler}--\eqref{BCs0} is compatible with this notion of symmetry. 

\begin{thm}\label{sympreserve}
    Let $(M,\mathsf{g})$ be a compact and connected Riemannian manifold (possibly with boundary), and let $(u,p)$ be a smooth solution to \eqref{Euler}--\eqref{BCs0} on $[0,T)$. Suppose $G$ is a Lie group acting smoothly on $M$, and $\mathsf{g}$ is $G$-invariant. The vector field $u(t)$ and the scalar field $p(t)$ are $G$-invariant for all $t\in [0,T)$ if and only if $u(0)$ is $G$-invariant. 
\end{thm}
\begin{proof}
    If $(u,p)$ solves \eqref{Euler}--\eqref{BCs0}, then for each $g\in G$, Proposition \ref{Eulergeometric} can be applied to the isometry $\varphi_g:M\to M$ to produce another solution $(\tilde{u}_g,\tilde{p}_g)$. If $u(0)$ is $G$-invariant then $u(0)=\tilde{u}_g(0)$, i.e., the initial conditions coincide. By the uniqueness assertion of Theorem \ref{Eulerwellposed}, we find that $\tilde{u}_g(t)=u(t)$ for all $t\in[0,T)$. It follows from \eqref{Euler} that $\nabla \tilde{p}_g(t)=\nabla p(t)$ for all $t\in[0,T)$, and in particular, since $M$ is connected, there is a function $c:G\to \mathbb{R}$ such that $\tilde{p}_g(t)=p(t)+c(g)$ for all $t\in[0,T)$. It is clear that $c:G\to \mathbb{R}$ is smooth, and furthermore, it is a group homomorphism (with the additive group structure on $\mathbb{R}$) since for all $g_1,g_2\in G$ and $x\in M$, we have 
    \begin{align*}
        c(g_1\cdot g_2)&=\tilde{p}_{g_1\cdot g_2}(t,x)-p(t,x)\\
        &=p (t,(g_1\cdot g_2) \cdot x)-p(t,x)\\
        &=p(t,g_1\cdot (g_2\cdot x))-p(t,x)\\
        &=\tilde{p}_{g_1}(t,g_2\cdot x)-p(t,x)\\
        &=p(t,g_2\cdot x)+c(g_1)-p(t,x)\\
        &=\tilde{p}_{g_2}(t,x)+c(g_1)-p(t,x)\\
        &=c(g_2)+c(g_1). 
    \end{align*}
    It follows that $c(G)$ is a compact subgroup of $\mathbb{R}$, and is therefore trivial. This shows that if $u(0)$ is invariant, then the same is true for $u(t)$ and $p(t)$ for all $t\in [0,T)$. The converse is obvious. 
\end{proof}

\section{Invariant geometry}
In Section \ref{PrelimI}, we saw that if the compact Lie group $G$ is acting smoothly and isometrically on a connected and compact Riemannian manifold $(M,\mathsf{g})$, then the incompressible Euler fluid equations \eqref{Euler}--\eqref{BCs0} are well-posed within the class of velocity-pressure pairs $(u,p)$ that are $G$-invariant. The purpose of this section is to describe the $G$-invariant Riemannian metrics, vector fields and scalar functions on $M$ in the special cases where $G$ acts on $(M,\mathsf{g})$ transitively or with cohomogeneity one (i.e., with at least one orbit that has codimension one in $M$). For this section and the next, we will assume that the compact manifold $(M,\mathsf{g})$ does \textit{not} have boundary, since this is the situation in the statement of Theorem \ref{maintheorem}.  

\subsection{Homogeneous spaces}
We begin by studying the case where $(M,\mathsf{g})$ is a homogeneous Riemannian manifold, i.e., $G$ acts transitively and isometrically. The main purpose of this subsection is to describe the $G$-invariant vector fields. Most of this material is standard and can be found in, for example, \cite[Chapter 7]{Besse}, \cite[Chapter 6]{AlexandrinoBettiol15}, and \cite[Appendix B]{BohmLafuente23}.

Choose any point $p\in M$, and let $H\subseteq G$ denote its isotropy subgroup. In the homogeneous case, $M$ is equivariantly diffeomorphic to $G/H$. Denote by $\fG$ and $\fH$ the Lie algebras of $G$ and $H$, respectively. Since $G$ is compact, we can equip it with a bi-invariant Riemannian metric $Q$, which induces an $\Ad(G)$-invariant inner product on $\mathfrak{g}$, i.e.,  $Q(\Ad(g)x,\Ad(g)y)=Q(x,y)$ for all $x,y\in\fG$ and $g\in G$. Let $\fM$ be the $Q$-orthogonal complement of $\fH$ inside $\fG$. For each $X\in\fG$, we can consider the \emph{action field} $X^*$ defined by $X^*(gH)=\frac{d}{dt}|_{t=0}(\exp(tX)\cdot g)H\in T_{gH}(G/H)$. Using action fields, we can identify $\fM$ with $T_{eH}(G/H)$ by sending $X$ to $X^*(eH)$.

There is a relationship between action fields and $G$-invariant vector fields on $G/H$. By Definition~\ref{Ginvariant}, a $G$-invariant vector field is completely determined by its value at $eH$, so every $G$-invariant vector field corresponds to a vector in $\fM\simeq T_{eH}(G/H)$. However, it is not true that every vector in $\fM$ corresponds to a $G$-invariant vector field.

\begin{lem}\label{hominvvec}
   Let $\psi:\fM\to T_{eH}(G/H)$ be such that $\psi(X)=X^*(eH)$. If $X\in\fX^G(G/H)$, then $\psi^{-1}(X(eH))\in\fM_0$, where $\mathfrak{m}_0\subseteq \mathfrak{m}$ is the largest subspace on which $\Ad(H)$ acts trivially.  Moreover, if $Y\in\fM_0$, then there exists $X\in\fX^G(G/H)$ such that $X(eH)=\psi(Y)$.
\end{lem}
\begin{proof}
    In order for $X\in \fX^G(G/H)$ to be well-defined, $X(eH)$ must be fixed by the action of $H$. The map $\psi$ is an intertwiner between the isotropy representation $H\ni h\mapsto (\rD\varphi_h)_e$ on $T_{eH}(G/H)$ and the adjoint representation $H\ni h\mapsto\Ad(h)$ on $\fM$, so $\psi^{-1}(X(eH))$ must be fixed by $\Ad(H)$ in $\fM$. Conversely, given $Y\in\fM_0$, the vector $\psi(Y)\in T_{eH}(G/H)$ is fixed by the action of $H$, so extends to a $G$-invariant vector field on $G/H$.
\end{proof}
\begin{rmk}
    In general, the unique $G$-invariant vector field $X$ with $X(eH)\in \mathfrak{m}_0$ does not coincide with $\psi^{-1}(X(eH))$, except at $eH$. 
\end{rmk}

We conclude this section by observing that when $G$ is compact, $G$-invariant vector fields on $G/H$ are always divergence-free. 
\begin{lem}\label{divhom}
    Let $X\in\fX^G(G/H)$. If $\sg$ is $G$-invariant, then $\nabla\cdot X=0$.
\end{lem}
\begin{proof}
    For a $G$-invariant vector field $X$, the scalar function $\nabla\cdot X$ is itself $G$-invariant (and therefore constant on $G/H$). The divergence theorem then gives the result. 
\end{proof}

\subsection{Principal part of cohomogeneity one manifolds}

We now extend the theory of the previous subsection to Riemannian manifolds $(M,\mathsf{g})$ where $G$ is acting with cohomogeneity one, i.e., there is an orbit with codimension one. This material is also standard and can be found in, for example, Chapters 3 and 6 of \cite{AlexandrinoBettiol15}.

For each $x\in M$, use $G_x\subseteq G$ to denote the isotropy subgroup. Unlike in the homogeneous case, isotropy groups can vary considerably.  
\begin{defin}\label{orbittypes}
  The orbit $Gx$ is said to be:
    \begin{itemize}
        \item \emph{principal} if for all $y\in M$, there exists $g\in G$ such that $G_x\subseteq gG_yg^{-1}$;
        \item \emph{singular} if $\dim(Gx)<\dim(Gy)$ for a principal orbit $Gy$;
        \item \emph{exceptional} if it is nonprincipal and nonsingular.
    \end{itemize}
\end{defin}

To describe the structure of $M$, we will use the slice theorem (see \cite[Theorem 3.57]{AlexandrinoBettiol15}). The slice theorem states that for every orbit $Gx$, there is a $G$-invariant neighbourhood of $Gx$ that is $G$-equivariantly diffeomorphic to the twisted product $G\times_{G_x}D$, where $D$ is a ball in the normal space $V$ to the orbit at $x$, the isotropy group $G_x$ acts linearly on $D$ (in particular, via the slice representation), and $G$ acts on the product by left multiplication in the first coordinate.

In the case where $Gx$ is a principal orbit, the action of the isotropy $G_x=H$ on $V$ is trivial (and $\dim V=1$). In the case where $Gx$ is nonprincipal, the isotropy $G_x=K$ acts transitively on spheres in $V$. If $Gx$ is exceptional, then $\dim V=1$, otherwise $\dim V>1$. This yields the following classification of the orbit space $M/G$ when $M$ is a closed manifold.

\begin{thm}\label{Ch1Structure}
If the connected and closed Riemannian manifold $(M,\mathsf{g})$ admits an isometric and cohomogeneity one action of a Lie group $G$, then $M/G$ is diffeomorphic to $S^1$ or $[0,1]$. In the first case, all orbits are principal. In the second case, $(0,1)$ corresponds to principal orbits, whereas the boundary points correspond to singular or exceptional orbits.
\end{thm}

We will now study $G$-invariant metrics on closed and connected cohomogeneity one manifolds. For now, we will focus on the geometry on the set $M_p$ of points of $M$ that lie on principal orbits, with nonprincipal orbits being discussed in the next subsection.

Let $\sg$ be a $G$-invariant metric on $M$. At each $x\in M$, consider the splitting of $T_xM$ into the tangent space $T_x(Gx)$ and normal space $N_x(Gx)$ of the orbit~$Gx$.

\begin{defin}\label{vert_hor} A smooth curve $\alpha:[a,b]\to M$ is
    \begin{itemize}
        \item \emph{vertical} if $\alpha'(r)\in T_{\alpha(r)}(G\alpha(r))$ for all $r\in[a,b]$, and 
        \item \emph{horizontal} if $\alpha'(r)\in N_{\alpha(r)}(G\alpha(r))$ for all $r\in[a,b]$.
    \end{itemize}
    A smooth vector field $u\in\fX(M)$ is
    \begin{itemize}
        \item \emph{vertical} if $u(x)\in T_x(Gx)$ for all $x\in M$, and 
        \item \emph{horizontal} if $u(x)\in N_x(Gx)$ for all $x\in M$. 
    \end{itemize}
\end{defin}

Let $\gamma:[0,1]\to M$ be a horizontal geodesic that intersects all of the orbits exactly once (except possibly at $\gamma(0)$ and $\gamma(1)$) and such that $G\gamma(r)$ is a principal orbit for all $r\in(0,1)$. Such a geodesic exists because if we start a geodesic at a point on a principal orbit with initial velocity orthogonal to that orbit, then this geodesic remains orthogonal to every orbit it meets. The isotropy groups $G_{\gamma(r)}$ are all equal for $r\in(0,1)$. This common isotropy is called the \emph{principal isotropy group} and is denoted by $H$. If $M/G=[0,1]$, then $G\gamma(0)$ and $G\gamma(1)$ are nonprincipal orbits. If $M/G=S^1$, then $G\gamma(0)=G\gamma(1)$ is a principal orbit (and $G_{\gamma(0)}=H$), although it is not necessarily true that $\gamma(0)=\gamma(1)$. For now, we will focus on the case $M/G=[0,1]$, with the case $M/G=S^1$ discussed at the end of this subsection.

Consider the horizontal vector field $\gamma'(r)$ along $\gamma$ for $r\in(0,1)$. This can be extended to a horizontal $G$-invariant vector field $\partial_r$ on $M_p$ by the action of $G$:
$$\partial_r(g\gamma(r))=\rD\varphi_g\gamma'(r).$$
This is well-defined since $\sg$ is $G$-invariant, so acting by $G$ preserves the normal space. The splitting of $T_xM$ at every point $x\in M_p$ into horizontal and vertical directions is then given by
$$T_xM=\spn\{\partial_r\}\oplus T_x(Gx).$$
Let $\rD r$ be the covector field corresponding to $\partial_r$. The following lemma describes all $G$-invariant metrics on $M_p$ in terms of invariant metrics on the homogeneous space fibres. 
\begin{lem}
    Let $H$ be the principal isotropy group. Up to scale, the $G$-invariant metric $\sg$ on $M_p$ can be written as
    \begin{equation}\label{metricform}
        \sg=\rD r^2+\sg_r,
    \end{equation}
    where $\sg_r$ is a smoothly-varying one-parameter family of $G$-invariant metrics on the homogeneous space $G/H$ for $r\in(0,1)$.
\end{lem}
From now, we will assume that $G$-invariant metrics always have the form \eqref{metricform}. Clearly, for the purposes of proving Theorem \ref{Eulerwellposed}, fixing the scale does not cause any loss of generality. Now, we will compute the Levi-Civita connection. First, if $X$ and $Y$ are both vertical vector fields, then since the principal orbits have codimension one, $\nabla_XY$ splits into horizontal and vertical components
\begin{equation}\label{splitLC}
    \nabla_XY=\sg(\nabla_XY,\partial_r)\partial_r+\nabla^r_XY,
\end{equation}
where $\nabla^r$ is the Levi-Civita connection of $\sg_r$ on the orbit $G\gamma(r)=G/H$. Second, it is clear that $\nabla_{\partial_r}\partial_r=0$ because $\gamma(r)$ is a geodesic. Finally, if $X$ is a vertical vector field, then the mixed connection term $\nabla_{X}\partial_r$ encodes the extrinsic geometry of the orbits. 
\begin{defin}\label{shapeoperatordef}
    For each $r\in(0,1)$, the \emph{shape operator} $S_r:\fX(G\gamma(r))\to\fX(G\gamma(r))$ is the linear operator defined by
    $$S_rX=-\nabla_X \partial_r.$$
\end{defin}

\begin{lem}\label{shapeoperator}
    Let $X,Y\in\fX(G\gamma(r))$. Then $S_r$ is $\sg_r$-symmetric and
    $$\sg_r(S_rX,Y)=\sg(\nabla_XY,\partial_r).$$
\end{lem}

The shape operator allows us to define the mean curvature, which measures the extrinsic curvature of the orbits.

\begin{defin}
    For each $r\in(0,1)$, the \emph{mean curvature} $\cH(r)$ of the orbit $G\gamma(r)$ is the trace of the shape operator. That is, if $\{E_i\}_{i=1}^n$ is a $\sg_r$-orthonormal basis of $\mathfrak{m}\simeq T_{\gamma(r)}(G\cdot \gamma(r))$, then at the point $\gamma(r)$, we have 
    $$\cH(r)=\sum_{i=1}^{n}\sg_r(S_rE_i,E_i).$$
\end{defin}

Since the mean curvature is constant on each orbit, the first variation formula for submanifolds gives us the following result. 
\begin{lem}\label{MCV}
    The mean curvature is given by
    $$\cH(r)=-\frac{1}{2}\frac{d}{dr}(\ln \vol_{\sg_r}(G/H)).$$
\end{lem}

Finally, to conclude our discussion about the case $M/G=[0,1]$, we note that if the orbit corresponding to $r=0$ is singular, the mean curvature $\cH(r)$ grows without bound as $r\to 0$ (see \cite[Proposition 1]{Pacini03}).

\begin{lem}\label{horsmoothness}
    Suppose $G\gamma(0)$ is a singular orbit. For any $G$-invariant Riemannian metric $\mathsf{g}$ on $M$, we have $\lim_{r\to 0}\vol_{\sg_r}(G/H)=0$. 
\end{lem}

If $M/G=S^1$, then much of the above analysis still holds, but now all orbits are principal. In particular, we can use the geodesic $\gamma:[0,1]\to M$ to construct a globally defined unit vector field $\partial_r$ which is perpendicular to every orbit, and the $G$-invariant Riemannian metric $\sg$ has the form \eqref{metricform}, where $\sg_r$ is defined for all of $r\in[0,1]$. In order for the metric to remain well-defined and smooth, $\{\sg_r\}_{r\in[0,1]}$ must satisfy some periodic boundary conditions. Indeed, let $a\in G$ be such that $\gamma(0)=a\gamma(1)$. Since $\sg_0$ is $G$-invariant, we have $\sg_0=\rD\varphi_a\sg_1$. However, since $\sg_1$ is also $G$-invariant, we have that $\rD\varphi_a\sg_1=\sg_1$, so $\{\sg_r\}_{r\in[0,1]}$ must satisfy $\sg_0=\sg_1$. By extending $\gamma$ to a horizontal geodesic $\gamma:(-\e,1+\e)\to M$ such that $G\gamma(r)=G\gamma(1+r)$ for $r\in(-\e,\e)$ (which we can do since $G$ acts by isometries), we can use a similar argument to conclude that any possible extension of $\sg_r$ to $r\in(-\e,1+\e)$ must satisfy $\sg_r=\sg_{1+r}$ for $r\in(-\e,\e)$. Since $\sg_r$ is a smooth family of metrics, it follows that the derivatives of all orders of $\sg_0$ and $\sg_1$ also coincide.

Next, we will characterise $G$-invariant vector fields on the principal part of the manifold $M_p$.
\begin{lem}\label{vectorlem}
    Denote $n_0=\dim\fM_0$. There exist finitely many smooth vertical $G$-invariant vector fields $\{Y_i\}_{i=1}^{n_0}$ on $M_p$ that are everywhere linearly independent and are such that every smooth vertical $G$-invariant vector field $v$ on $M_p$ can be written as
    $$v=\sum_{i=1}^{n_0} v_i(r)Y_i,$$
    where the functions $v_1,\dots,v_{n_0}:(0,1)\to\R$ are smooth. Moreover, a smooth vector field $u$ on $M_p$ is $G$-invariant if and only if it has the form
    \begin{equation}\label{vectorform}
      u=h(r)\partial_r+v,  
    \end{equation}
    where $h:(0,1)\to\R$ and $v$ is a smooth vertical $G$-invariant vector field on $M_p$.
\end{lem}
\begin{proof}
We will assume that $M_p$ is equipped with a $G$-invariant metric of the form \eqref{metricform}. If $M/G=[0,1]$, define $\{Y_i\}_{i=1}^{n_0}$ to be the $G$-invariant vector fields induced by a fixed basis of $\mathfrak{m}_0 \subseteq T_{\gamma(r)}(G\gamma(r))$. If $M/G=S^1$, then $M_p=M$, and we instead choose a smooth function $\beta:(-\varepsilon,1+\varepsilon)\to(\mathfrak{m}_0)^{n_0}$ such that $\{\beta_i(r)\}_{i=1}^{n_0}$ is a basis for $\mathfrak{m}_0$ for each $r$, and $\beta_i(r)=\mathrm{d}\varphi_a(\beta_i(r+1))$ for $r\in(-\varepsilon,\varepsilon)$. The vector fields $\{Y_i\}_{i=1}^{n_0}$ are then obtained by extending $\beta(r)$ along $\gamma$ to all of $M$, which is well-defined by the monodromy condition on $\beta$. The claim then follows since $\{Y_i\}$ forms a basis for $G$-invariant vector fields on each homogeneous space fibre $G/H\simeq G\gamma(r)$.

For the second claim, it is clear that any vector field of the form \eqref{vectorform} is $G$-invariant. Conversely, let $u\in\mathfrak{X}^G(M_p)$, and split $u(\gamma(r))=h(r)\gamma'(r)+v(r)$ into horizontal and vertical components. Since $u$ is $G$-invariant, we have $u(g\gamma(r))=h(r)\partial_r + \mathrm{d}\varphi_g v(r)$, and since $G$ acts by isometries, $\mathrm{d}\varphi_g v(r)\in T_{g\gamma(r)}(G\gamma(r))$, so $v\in\mathfrak{X}^G(G/H)$. 
Decomposing $u(\gamma(r))=h(r)\gamma'(r)+v(r)$ into horizontal and vertical parts, $G$-invariance of $u$ forces $\rD\varphi_g v(r)=v(g\gamma(r))$, so $v$ is $G$-invariant.
\end{proof}

Now, we compute the Levi-Civita connection and divergence of $G$-invariant vector fields.

\begin{lem}\label{LCu}
    The vector field $u$ from  \eqref{vectorform} satisfies 
    \begin{equation}\label{LCCH1}
        \nabla_{u}u=\left(hh'+\sg_r(S_r v,v)\right)\partial_r-h S_r(v)+h\nabla_{ \partial_r}v+\nabla_v^r v
    \end{equation}
    and
    \begin{equation}\label{divCH1}
        \nabla\cdot u=h'- \cH h,
    \end{equation}
    where $'$ denotes an $r$ derivative.
\end{lem}
\begin{proof}
For the Levi-Civita connection, we have
\begin{align*}
    \nabla_u u&=\nabla_{h\partial_r}(h\partial_r)+\nabla_v(h\partial_r)+\nabla_{h\partial_r}v+\nabla_vv\\
    &=h(\partial_r h)\partial_r+h^2\underbrace{\nabla_{\partial_r}\partial_r}_{=0}+\underbrace{(vh)}_{=0}\partial_r+h\nabla_v\partial_r+h\nabla_{\partial_r}v+\nabla_vv\\
    &=hh'\partial_r-hS_r v+h \nabla_{\partial_r}v+\sg_r(S_r v,v)\partial_r+\nabla^r_vv.
\end{align*}
In this computation, we used Definition \ref{shapeoperatordef}, Lemma \ref{shapeoperator}, the fact that $\gamma$ is a geodesic and $G$ acts by isometries. Equation \eqref{divCH1} follows from Lemmas B.3 and B.4 of \cite{BohmLafuente23}. 
\end{proof}

\subsection{Smoothness of invariant vector fields near nonprincipal orbits}\label{singPrelims}

So far, we have discussed $G$-invariant vector fields on the principal part $M_p$ of $M$. Extending this theory to nonprincipal orbits (which only occur at boundary points of $M/G$) is more complicated, since we require the vector field to be invariant under a larger isotropy group at the nonprinicpal orbit while simultaneously being smooth at the nonprincipal orbit. In this subsection, we use the slice theorem to describe how to smoothly extend vector fields that are invariant on $M_p$ to be invariant on the nonprincipal orbit.

The slice theorem states that the local structure of a cohomogeneity one manifold near a nonprincipal orbit is the same as the manifold $G\times_{K}V$, where $K$ is the isotropy at the nonprincipal orbit and acts irreducibly and linearly on the Euclidean vector space $V$ and transitively on all concentric spheres centered at the origin. A $G$-invariant vector field $u$ on $G\times_K V$ is determined by its value at points $[(e,v)]$, $v\in V$. Letting $\mathfrak{n}$ be an $\Ad(K)$-invariant complement of $\mathfrak{k}=\mathrm{Lie}(K)$ in $\fG$, the slice theorem also implies that $T_{[(e,v)]}(G\times_K V)\simeq V\oplus\mathfrak{n}$. It follows that $u$ is determined by the $K$-equivariant map $W:V\to V\oplus \mathfrak{n}$, defined so that $W(v)$ is the value of $u([e,v])$ under the identification $T_{[(e,v)]}(G\times_K V)\simeq V\oplus\mathfrak{n}$.

The following lemma allows us to express the map $W$ as a combination of \emph{finitely many} homogeneous polynomials (that is, maps $P:V\to V\oplus\mathfrak{n}$ of the form $P(v)=Q(v,\dots,v)$, where $Q:V^{d}\to V\oplus \mathfrak{n}$ is a symmetric multilinear map, and $d$ is some nonnegative integer known as the \emph{degree} of $P$).
\begin{lem}\label{EWSmoothness}\
    \begin{enumerate}[(i)]
        \item There are finitely many homogeneous and $K$-equivariant polynomials \\ $W_0,\dots,W_N:V\to V\oplus \mathfrak{n}$ such that any smooth $G$-invariant vector field $u$ on $G\times_K V$ evaluated at points $[e,v]$, $v\in V$ is of the form
        \begin{equation}\label{polyform}
            u([e,v])=\sum_{i=0}^N u_i(|v|^2) W_i(v),
        \end{equation}
        where $u_i:V \to \mathbb{R}$ are smooth functions of only $|v|^2$. Moreover, all such vector fields of the form \eqref{polyform} are smooth and $G$-invariant.
        \item The polynomials $W_i$ in \eqref{polyform} can always be arranged such that $W_0$ generates the horizontal vector field $r\partial_r$ and $W_1,\dots,W_N$ generate vertical vector fields on $G\times_{K}V$.
        \item If $W_i$ is a degree $0$ polynomial in \eqref{polyform}, then $\mathrm{im}(W_i)\subseteq\{0\}\oplus\mathfrak{n}_0$, where $\mathfrak{n}_0$ is the largest subspace of $\mathfrak{n}$ on which $\Ad(K)$ acts trivially.
    \end{enumerate}  
\end{lem}
\begin{proof}\
\begin{enumerate}[(i)]
    \item The proof is analogous to the proof of \cite[Lemma 1.1]{EschenburgWang}, except that the $K$-equivariant maps are into $V\oplus\mathfrak{n}$ instead of $\mathrm{Sym}^2(V\oplus\mathfrak{n})$. In particular, to see that we can choose finitely many generators $W_0,\dots,W_N$, let $\mathcal{W}$ be the space of all $K$-equivariant maps from the unit sphere in $V$ to $V\oplus\mathfrak{n}$. The evaluation map then gives an isomorphism $\mathcal{W}\simeq (V\oplus\mathfrak{n})^K$, so $\mathcal{W}$ is finite-dimensional. Define $\mathcal{W}_p$ to be the subspace of all maps that are restrictions to the unit sphere of $K$-equivariant homogeneous polynomials $V\to V\oplus\mathfrak{n}$ of degree $p$, and $\mathcal{W}^m=\sum_{p=0}^m \mathcal{W}_p$. By polynomial approximation, $\bigcup_m \mathcal{W}^m$ is dense in $\mathcal{W}$, and since $\mathcal{W}$ is finite-dimensional we get $\mathcal{W}=\mathcal{W}^{m_0}$ for some $m_0$. We can then choose a basis $W_0,\dots,W_N$ adapted to the filtration
    $$\mathcal{W}^0\subseteq \mathcal{W}^1\subseteq\cdots \subseteq \mathcal{W}^{m_0}=\mathcal{W}^{m_0+1}=\cdots=\mathcal{W}.$$
    \item First, observe that the $K$-equivariant homogeneous degree one polynomial $W:V\to V\oplus\mathfrak{n}$, $v\mapsto (v,0)$, corresponds to the vector field $r\partial_r$. Next, since $V$ and $\mathfrak{n}$ are both $K$-invariant subspaces of $V\oplus\mathfrak{n}$, the $K$-equivariant homogeneous polynomial $W_i$ in \eqref{polyform} splits as
    $$W_i=W_i^V+W_i^\mathfrak{n},$$
    where $W_i^V:V\to V$ and $W_i^\mathfrak{n}:V\to\mathfrak{n}$ are themselves $K$-equivariant homogeneous polynomials. Vector fields corresponding to the polynomials $W_i^\mathfrak{n}$ are vertical since we can identify $\mathfrak{n}$ with the tangent space to the nonprincipal orbit at $\gamma(0)$. We now split into two cases based on the degree $d_i$ of $W_i^V$.
    \begin{itemize}
        \item If $d_i$ is even, then the $K$-invariant homogeneous polynomial $V\to\R$, $v\mapsto \ip{W_i^V(v)}{v}$, has degree $d_i+1$, which is odd. However, all such polynomials are zero.
        It follows that $\ip{W_i^V(v)}{v}=0$ for all $v\in V$, so $W_i^V$ is vertical (i.e., tangent to spheres in $V$). Overall, $W_i$ is vertical.
        \item If $d_i$ is odd, then the $K$-invariant scalar polynomial $v\mapsto \ip{W_i^V(v)}{v}$ has even degree $d_i+1$, and so must satisfy ${v}=\alpha_i|v|^{d_i+1}$ for some $\alpha_i\in\R$. The $K$-equivariant homogeneous polynomial $W_i'=W_i-\alpha_i|v|^{d_i-1} W$ is vertical, since
        $$\ip{(W_i')^V(v)}{v}=\ip{W_i^V(v)-\alpha_i|v|^{d_i-1}v}{v}=\alpha_i|v|^{d_i+1}-\alpha_i|v|^{d_i-1}\ip{v}{v}=0.$$
        Adding $W$ and replacing $W_i$ with $W_i'$ in the finite collection of generators in part (i) does not affect the span, since $W_i'$ is given by a linear combination of $W_i$ and $W$ with coefficients given by smooth functions of $|v|^2$.
    \end{itemize}
    Applying this procedure to each $W_i$ yields a collection of vertical vector fields.
    \item If $W_i$ is a homogeneous degree zero polynomial $V\to V\oplus\mathfrak{n}$, then $W_i$ is a constant map $v\mapsto w_i\in V\oplus\mathfrak{n}$. Since $W_i$ is also $K$-equivariant, we must have that $K$ acts trivially on $w_i$. In particular, since $K$ acts transitively on spheres in $V$, the only element of $V$ that is fixed by all of $K$ is $0$. Thus, $w_i\in\{0\}\oplus\mathfrak{n}_0$.
\end{enumerate}
\end{proof}
\begin{cor}\label{Christoffel}
    If $\mathsf{g}$ is a $G$-invariant Riemannian metric on $M$, then $\nabla_{W_i}W_j=\sum_{k=0}^{N}c_{ij}^{k}W_k$, where $c_{ij}^k:V\to\R$ are smooth functions of $|v|^2$. 
\end{cor}
\begin{proof}
    If $\sg$ is a smooth $G$-invariant metric and $W_i,W_j$ are smooth $G$-invariant vector fields, then $\nabla_{W_i}W_j$ is a smooth $G$-invariant vector field. Applying part (i) of Lemma \ref{EWSmoothness} gives the result.
\end{proof}

The slice theorem also states that the map $\Phi:V \oplus \mathfrak{n}\to G\times_{K}V$ with $(v,x)\mapsto [(\text{exp}(x),v)]$ is a $K$-equivariant local diffeomorphism at the origin $(0,0)\in V\oplus \mathfrak{n}$. It induces coordinates on $G\times_{K}V$ near a point on a nonprincipal orbit, known as \emph{slice coordinates}, and using slice coordinates allows us to think of vector fields as functions $X: V\oplus \mathfrak{n}\to V\oplus \mathfrak{n}$. The following lemma relates the homogeneous polynomials in Lemma \ref{EWSmoothness} to their corresponding vector fields $V\oplus\mathfrak{n}\to V\oplus\mathfrak{n}$ in slice coordinates.

\begin{lem}\label{vecfielddegree}
    If $W:V\to V\oplus \mathfrak{n}$ is a $K$-equivariant homogeneous polynomial of degree $k$, then the corresponding $G$-invariant smooth vector field $\hat{W}:V\oplus \mathfrak{n}\to V\oplus \mathfrak{n}$ in slice coordinates $(v,x)$ satisfies
    $$\hat{W}(v,x)=W(v)+O(|x|)\cdot O(|v|^k)$$
    in a neighborhood of the origin. 
\end{lem}
\begin{proof}
    The corresponding vector field $\tilde{W}$ on $G\times_K V$ is given at slice points $[(e,v)]\in G\times_K V$ by
    $$\tilde{W}([(e,v)])=\rD\Phi_{(v,0)}W(v).$$
    At other points $[(g,v)]$, we act by $G$ to obtain
    $$\tilde{W}([(g,v)])=\rD(\varphi_g)_{[(e,v)]} \tilde{W}([(e,v)]).$$
    This is well-defined since $W$ and $\Phi(\cdot,0)$ are $K$-equivariant (so $\rD\Phi_{(v,0)}$ is also $K$-equivariant). By the slice theorem, all elements near the nonprincipal orbit are of the form $[(g,v)]=[(\exp(x),v)]$ for some $x\in\mathfrak{n}$. In this case, we have
    $$\tilde{W}([(g,v)])=\rD(\varphi_{\exp(x)})_{[(e,v)]} \circ \rD\Phi_{(v,0)}W(v).$$
    Pulling back to our slice coordinates $(v,x)$ on $V\oplus \mathfrak{n}$ through $\Phi$ gives
    $$\hat{W}(v,x)=A(v,x)W(v),$$
    where $A(v,x)=\rD\Phi_{(v,x)}^{-1} \circ \rD(\varphi_{\exp(x)})_{[(e,v)]} \circ \rD\Phi_{(v,0)}$ is linear and smooth in $(v,x)$. Moreover, when $x=0$, we have $A(v,0)=\id_{V\oplus\mathfrak{n}}$. Taylor expanding $X$ at $x=0$ thus gives
    $$\hat{W}(v,x)=W(v)+C(v,x)W(v),$$
    where $C(v,x)=A(v,x)-\id_{V\oplus\mathfrak{n}}:V\oplus\mathfrak{n}\to V\oplus\mathfrak{n}$ is a linear operator satisfying $C(v,0)=0$ and $\|C(v,x)\|_{\mathcal{L}(V\oplus \mathfrak{n})}=O(|x|)$. Since $W$ is homogeneous of degree $k$, we have that $W(v)=O(|v|^k)$, which implies that $C(v,x)W(v)=O(|x|)\cdot O(|v|^k)$. 
\end{proof}

It is also helpful to identify how the polynomial structure changes when taking connections. For a given $K$-equivariant map $X:V\to V\oplus \mathfrak{n}$, we let $X^{(k)}$ denote the $k^\text{th}$ degree homogeneous polynomial in the Taylor expansion. For the corresponding invariant vector field $\hat{X}:V\oplus\mathfrak{n}\to V\oplus \mathfrak{n}$, we define $\hat{X}^{(k)}=\widehat{X^{(k)}}$.

\begin{lem}\label{ivg}
    Let $Z,W:V\to V\oplus\mathfrak{n}$ be $K$-equivariant homogeneous polynomials of degree $0$ and $k$ respectively. If $\hat{Z}$ and $\hat{W}$ are the corresponding vector fields in slice coordinates, then 
    \begin{align*}
        (\nabla_{\hat{Z}}\hat{W})^{(\ell)}=(\nabla_{\hat{W}}\hat{Z})^{(\ell)}=0 
    \end{align*}
    for $\ell=0,\dots,k-1$.
\end{lem}
\begin{proof}
    Let $(v^i)_{i=1}^\alpha$ be linear coordinates on $V$ and $(x^i)_{i=1}^\beta$ be linear coordinates on $\mathfrak{n}$. We then put slice coordinates $(y^i)_{i=1}^{\alpha+\beta}$ on $V\oplus\mathfrak{n}$ with $y^i=v^i$ for $i\leq\alpha$ and $y^i=x^{i-\alpha}$ for $i>\alpha$.
    Since $\sg$ is $G$-invariant, $\nabla_{\hat{Z}}\hat{W}$ and $\nabla_{\hat{W}}\hat{Z}$ are $G$-invariant vector fields, so they are determined by their values on the slice $x=0$. Lemma \ref{vecfielddegree} gives $\hat{Z}(v,0)=Z(v)$ and $\hat{W}(v,0)=W(v)$ on the slice, so in slice coordinates $(y^i)$ we have
    \begin{align*}
        \left.(\nabla_{\hat{Z}}\hat{W})^i\right|_{x=0} &= Z^j(v)\partial_j \hat{W}^i(v,0) + \Gamma^i_{jm}(v,0)Z^j(v)W^m(v), \\
        \left.(\nabla_{\hat{W}}\hat{Z})^i\right|_{x=0} &= W^j(v)\partial_j \hat{Z}^i(v,0) + \Gamma^i_{jm}(v,0)W^j(v)Z^m(v).
    \end{align*}
    In both expressions, the second term is $O(|v|^k)$, since $\Gamma^i_{jm}$ is smooth and bounded, $Z^j(v)$ and $Z^m(v)$ are $O(1)$, and $W^j(v)$ and $W^m(v)$ are $O(|v|^k)$.

    For the first term of $(\nabla_{\hat{Z}}\hat{W})^i|_{x=0}$, we split over coordinate directions. For $j\leq\alpha$, by part (iii) of Lemma \ref{EWSmoothness}, $Z(v)\in\{0\}\oplus\mathfrak{n}_0$, so $Z^j(v)=0$ (and the first term is zero) for all $j\leq\alpha$. For $j>\alpha$, Lemma \ref{vecfielddegree} gives us that $\hat{W}(v,x)=W(v)+O(|x|)O(|v|^k)$, so $\partial_j \hat{W}^i(v,0)=O(|v|^k)$. Since $Z(v)=O(1)$, the first term is $O(|v|^k)$.

    For the first term of $(\nabla_{\hat{W}}\hat{Z})^i|_{x=0}$, by Lemma \ref{vecfielddegree}, we have $\hat{Z}(v,x)=Z(v)+O(|x|)$. Therefore, $\partial_j\hat{Z}^i(v,0)=O(1)$. Since $W(v)=O(|v|^k)$, the first term is $O(|v|^k)$.

    Combining, $(\nabla_{\hat{Z}}\hat{W})^i|_{x=0}=O(|v|^k)$ and $(\nabla_{\hat{W}}\hat{Z})^i|_{x=0}=O(|v|^k)$, giving $(\nabla_{\hat{Z}}\hat{W})^{(\ell)}=(\nabla_{\hat{W}}\hat{Z})^{(\ell)}=0$ for all $\ell=0,\dots,k-1$.
\end{proof}

\begin{rmk}
    If $Z$ is instead a degree $p\ge 1$ polynomial, then a similar argument shows that both $\nabla_{\hat{Z}}\hat{W}$ and $\nabla_{\hat{W}}\hat{Z}$ are $O(|v|^{k+p-1})$ on the slice. That is, we have $(\nabla_{\hat{Z}}\hat{W})^{(\ell)}=(\nabla_{\hat{W}}\hat{Z})^{(\ell)}=0$ for $\ell=0,\dots,k+p-2$.
\end{rmk}

\section{Solutions of the Euler fluid equations}

In this section, we analyse highly symmetric solutions to the Euler fluid equation \eqref{Euler}--\eqref{BCs0} and prove Theorem \ref{maintheorem} using the blow-up criterion of Theorem \ref{Eulerwellposed}. First, we consider the problem on homogeneous manifolds, in which case the space of invariant vector fields is finite-dimensional, so global regularity follows from conservation of energy. Next, we consider the problem on cohomogeneity one manifolds with nonprincipal orbits (i.e., when $M/G=[0,1]$). In this case, we show that any divergence-free vector field is necessarily vertical, from which we deduce a certain ``decoupling'' of the Euler fluid equations into problems on the individual orbits; regularity then follows from local analysis. Finally, we consider cohomogeneity one manifolds with $M/G=S^1$. While the equations no longer ``decouple'', the fact that all orbits are principal allows us to directly bound solutions in $C^1$ using the conservation of energy and the scalar maximum principle.

\subsection{Homogeneous manifolds}

Let $\sg$ be a $G$-invariant metric on the homogeneous space $G/H$ with $G$ compact. By Lemma \ref{divhom}, if $u\in\fX^G(G/H)$, then $u$ automatically satisfies $\nabla\cdot u=0$. Also, if $p:G/H\to\R$ is $G$-invariant, then it is constant on $G/H$ and $\nabla p=0$. Therefore, the Euler equations \eqref{Euler} for a $G$-invariant vector field $u$ and $G$-invariant pressure $p$ simplify to the single equation
\begin{equation}\label{Eulerhom}
    \frac{du}{dt}=-\nabla_u u.
\end{equation}

It is easy to show that smooth solutions to \eqref{Eulerhom} can be extended indefinitely in time.

\begin{prop}\label{noblow-upHom}
    For any smooth $G$-invariant vector field $u_0\in\fX^G(G/H)$, there exists a unique smooth solution $u:\mathbb{R}\to\fX^G(G/H)$ to \eqref{Eulerhom} such that $u(0,x)=u_0(x)$ for all $x\in G/H$. Moreover, $\|u(t,x)\|=\|u_0(x)\|$ for all $x\in G/H$ and $t\in\R$.
\end{prop}
\begin{proof}
 If $u$ solves \eqref{Eulerhom}, then
    $$\frac{d}{dt}\sg(u,u)=2\sg\left(\frac{du}{dt},u\right)=-2\sg(\nabla_u u,u)=-u(\sg(u,u)).$$
    Since $\sg$ and $u$ are $G$-invariant, $\sg(u,u)$ is constant in space. It follows that
    $$\frac{d}{dt}\sg(u,u)=0,$$
    and $\sg(u,u)$ remains constant in time. Therefore, $\|u(t)\|_{C^0}$ is uniformly bounded in time. Since $\fX^G(G/H)$ is finite-dimensional, it follows that the $C^1$ norm of $u$ is also uniformly bounded, so the blow-up criterion of Theorem \ref{Eulerwellposed} implies that the solution exists for all $t\in \mathbb{R}$. 
\end{proof}

\subsection{Cohomogeneity one manifolds}
Now, let $\sg$ be a $G$-invariant metric on the cohomogeneity one manifold $M$. By Lemma \ref{LCu}, the incompressible Euler fluid equations for a vector field $u$ of the form \eqref{vectorform} and a scalar pressure field $p$ become
\begin{align}\label{EulerCH1}
\begin{split}
    \frac{dh}{dt}&=-hh'-\sg_r(S_rv,v)-p',\\
    \frac{dv}{dt}&=h S_rv-h\nabla_{\partial_r}v-\nabla^r_v v,
    \end{split}
\end{align}
with divergence-free condition 
\begin{align}\label{Div0}
0&=h'-\cH h.
\end{align}

Because the principal orbits have codimension one and the $G$-invariant vector fields in the vertical direction are automatically divergence-free, the divergence-free condition \eqref{Div0} becomes an ODE in $r$ for the horizontal component. As a result, the structure of the space of horizontal divergence-free $G$-invariant vector fields is simple to describe.

\begin{lem}\label{BCs}
    If $M/G=[0,1]$, then the only horizontal divergence-free $G$-invariant vector field is the zero vector field. If $M/G=S^1$, then the space of horizontal divergence-free $G$-invariant vector fields is one-dimensional.
\end{lem}
\begin{proof}
    Consider the case $M/G=[0,1]$. The expression for the mean curvature in Lemma \ref{MCV} reveals that the ODE \eqref{Div0} is solved by
    \begin{equation}\label{hsolution}
        h(r)=ch_0(r),\quad h_0(r)=\sqrt{\frac{\vol_{\sg_{1/2}}(G/H)}{\vol_{\sg_{r}}(G/H)}}.
    \end{equation}
    Suppose $G\gamma(0)$ is a singular orbit. By Lemma \ref{horsmoothness}, $\vol_{\sg_{r}}(G/H)\to0$ as $r\to0$. It follows that $h_0$ blows up as $r\to 0$, so in order for $h\partial_r$ to be a smooth vector field, we must have $c=0$. Thus, $h=0$.
    
    Next, suppose $G\gamma(0)$ is an exceptional orbit. By a similar argument to the proof of part (iii) of Lemma \ref{EWSmoothness}, we must have $h(0)=0$. Since $h_0\neq0$, it follows that $c=0$, so that $h=0$.

    Finally, if $M/G=S^1$, then all orbits are principal. The solution $h$ must satisfy a periodic boundary conditions at $r=0$ and $r=1$. Since the metric $\mathsf{g}_r$ is also periodic on $[0,1]$, every vector field of the form \eqref{hsolution} satisfies the periodic boundary condition.
\end{proof}

Lemma \ref{BCs} implies that in the case $M/G=[0,1]$, any divergence-free $G$-invariant vector field points solely in the vertical direction (i.e., tangent to the orbits with no ``mixing'' between orbits). We can then consider the problem at each orbit $G\gamma(r)$ separately as a pointwise homogeneous problem for each $r$. We thus get $C^0$ bounds immediately using a similar argument to Proposition \ref{noblow-upHom}. 
\begin{lem}\label{intervalC0}
   Suppose $M/G=[0,1]$. If $u:[0,T)\to \fX^G(M)$ is a smooth solution to \eqref{EulerCH1}--\eqref{Div0}, then for each $x\in M$ and $t\in [0,T)$, we have $\|u(t,x)\|=\|u(0,x)\|$. 
\end{lem}
\begin{proof}
    By Lemma \ref{BCs}, any $u\in\fX^G(M)$ points solely in the vertical direction, i.e., $h=0$. 
    The second equation of \eqref{EulerCH1} reveals that $v$ must satisfy $\frac{dv}{dt}=-\nabla_v^r v$, so
    $$\frac{d}{dt}\mathsf{g}_r(v(r),v(r))=- v(r)\mathsf{g}_r(v(r),v(r))=0$$
    for each $r\in (0,1)$. This implies that for each $r\in (0,1)$ and $t\in [0,T)$, we have $\|v(t,r)\|=\|v(0,r)\|$. The result thus follows unless $x$ is in a nonprincipal orbit, in which case using the fact that $v$ is smooth and taking the limit of $\|v(t,r)\|=\|v(0,r)\|$ as $r\to 0$ (or $r\to 1$) gives the result.
\end{proof}

Let us now turn to $C^1$ bounds in the case $M/G=[0,1]$. To do this, we will use the framework of Section \ref{singPrelims}. 
Consider a smooth time-varying $G$-invariant vertical vector field $u(t)$ restricted to an invariant neighbourhood of a nonprincipal orbit that is equivariantly diffeomorphic to $G\times_K V$. There is a corresponding time-varying $K$-equivariant map $W(t):V\to V\oplus\mathfrak{n}$ found by restricting $u$ to slice points $[(e,v)]\in G\times_K V$. We then decompose $W(t):V\to V\oplus\mathfrak{n}$ into a finite sum
$$W(t,v)=\sum_{i=1}^N u_i(t,|v|^2)W_i(v),$$
where all of the $W_i$ are vertical, by parts (i) and (ii) of Lemma \ref{EWSmoothness}. Taylor expanding the coefficients $u_i$ at $v=0$ to $1^\text{st}$ order in $|v|^2$ and splitting the sum into terms of the same degree gives
$$W(t,v)=\sum_{\deg(W_i)=0} u_i(t,0)W_i(v)+\sum_{\deg(W_i)=1} u_i(t,0)W_i(v)+R(t,v).$$
The first term is a time-varying function that is constant in $v$, which we will denote $W^{(0)}(t)\in V\oplus\mathfrak{n}$. The second term is a time-varying function that is linear in $v$, which we will denote $W^{(1)}(t)\in\mathcal{L}(V;V\oplus\mathfrak{n})^K$. The term $R(t,v)$ includes all the generators $W_i$ with degree at least $2$, as well as the remainder terms from the Taylor expansion of the coefficients $u_i$ (which have degree at least $2$ in $v$ as $u_i$ is a function of $|v|^2$). Finally, extending $W^{(0)}$, $W^{(1)}$ and $R$ back to invariant vector fields in slice coordinates $(v,x)\in V\oplus\mathfrak{n}$ (as in Lemma \ref{vecfielddegree}) yields
$$u(t,v,x)=u^{(0)}(t,v,x)+u^{(1)}(t,v,x)+u^{(2+)}(t,v,x),$$
with $u^{(0)}=\widehat{W^{(0)}}$, $u^{(1)}=\widehat{W^{(1)}}$, and $u^{(2+)}=\hat{R}$.

The following lemma gives a uniform bound for the first two terms in the Taylor expansion. To do this, we write an evolution equation for $u^{(1)}$ in terms of $u^{(0)}$ and $u^{(1)}$. The only terms that are nonlinear in $u^{(1)}$ are rescaled in a neighbourhood of the nonprincipal orbit and shown to satisfy the same energy cancellation law as in Proposition \ref{noblow-upHom}.

\begin{lem}\label{u1bounds}
 Suppose $M/G=[0,1]$. If $u:[0,T)\to \fX^G(M)$ is a smooth solution to \eqref{EulerCH1}--\eqref{Div0}, then $u^{(0)}$ and $u^{(1)}$ are both uniformly bounded on $[0,T)$ in a neighbourhood of a nonprincipal orbit.  
\end{lem}
\begin{proof}
    The boundedness of $u^{(0)}$ follows from Lemma \ref{intervalC0}. For $u^{(1)}$, observe that since taking the Taylor expansion of $u$ in the $V$ direction is independent of time, we have $(\frac{du}{dt})^{(1)}=\frac{du^{(1)}}{dt}$. Also, the map $u\mapsto u^{(1)}$ is linear because in the decomposition given in Lemma \ref{EWSmoothness}, the generators $W_i$ are fixed and the functions $u_i$ are coefficients with respect to the basis $\{W_i\}$ (so change linearly in $u$). By Lemma \ref{ivg} (and the subsequent remark), we have
    $$(\nabla^r_{u^{(0)}} u^{(2+)})^{(1)}=(\nabla^r_{u^{(2+)}} u^{(0)})^{(1)}=(\nabla^r_{u^{(1)}} u^{(2+)})^{(1)}=(\nabla^r_{u^{(2+)}} u^{(1)})^{(1)}=(\nabla^r_{u^{(2+)}} u^{(2+)})^{(1)}=0.$$
    It follows that
    $$\frac{d u^{(1)}}{dt}=(-\nabla^r_{u}u)^{(1)}=-(\nabla^r_{u^{(0)}} u^{(0)})^{(1)}-(\nabla^r_{u^{(0)}} u^{(1)})^{(1)}-(\nabla^r_{u^{(1)}} u^{(0)})^{(1)}-(\nabla^r_{u^{(1)}} u^{(1)})^{(1)}.$$
    We define a bilinear form $\tilde{\sg}$ on $\mathcal{L}(V,V\oplus \mathfrak{n})^{K}$ by
    $$\tilde{\mathsf{g}}(a,b)=\lim_{|v|\to 0}\frac{1}{|v|^2} \mathsf{g}_{[(e,v)]}(a(v),b(v)),$$
    where $\sg$ is the $G$-invariant metric on $G\times_K V$, and $a(v),b(v)$ are tangent vectors at $[(e,v)]$ via the identification $T_{[(e,v)]}(G\times_K V)\simeq V\oplus\mathfrak{n}$. Since $\sg$ is $G$-invariant, $\sg_{[(e,v)]}(a(v),b(v))$ is a $K$-invariant scalar function, thus depends only and smoothly on $|v|^2$ and is zero at $v=0$. Thus, the limit that defines $\tilde{\sg}$ indeed exists. We claim that $\tilde{\sg}$ is an inner product on $\mathcal{L}(V,V\oplus\mathfrak{n})^K$. It is clear that $\tilde{\sg}$ is symmetric and positive semidefinite. It is in fact positive definite: if $a\in\mathcal{L}(V,V\oplus\mathfrak{n})^K\setminus\{0\}$, then there is a unit vector $v_0\in V$ with $a(v_0)\neq0$. Taking $v=\lambda v_0$, we get
    $$\frac{1}{|v|^2}\sg_{[(e,v)]}(a(v),a(v))=\sg_{([e,\lambda v_0])}(a(v_0),a(v_0)).$$
    Taking the limit as $\lambda\to 0$ gives $\tilde{\sg}(a,a)=\sg_{[(e,0)]}(a(v_0),a(v_0))\neq0$, since $\sg_{[(e,0)]}$ is the metric at the nonprincipal orbit and is nondegenerate. Therefore, $\tilde{\sg}$ is an inner product.
    
    We claim that $(\nabla^r_{u^{(1)}} u^{(1)})^{(1)}$ satisfies the following with respect to $\tilde{\mathsf{g}}$:
    $$\tilde{\sg}\left((\nabla^r_{u^{(1)}} u^{(1)})^{(1)},u^{(1)}\right)=0.$$
    Indeed, since $u^{(1)}$ is $G$-invariant, we have 
    $$0=\mathsf{g}(\nabla_{u^{(1)}}u^{(1)},u^{(1)}).$$
    As we did with $u$, we now Taylor expand $\nabla_{u^{(1)}}u^{(1)}$ as $\nabla_{u^{(1)}}u^{(1)}=(\nabla_{u^{(1)}}u^{(1)})^{(1)}+(\nabla_{u^{(1)}}u^{(1)})^{(2+)}$. Observe that $(\nabla_{u^{(1)}}u^{(1)})^{(0)}=0$ by the remark after Lemma \ref{ivg}. Also, $(\nabla_{u^{(1)}}u^{(1)})^{(2+)}=O(|v|^2)$ and $u^{(1)}=O(|v|)$, so $\sg((\nabla_{u^{(1)}}u^{(1)})^{(2+)},u^{(1)})=O(|v|^3)$. Therefore, dividing by $|v|^2$ and taking the limit as $|v|\to 0$ causes this term to vanish. It follows that
    $$0=\lim_{|v|\to0}\frac{1}{|v|^2}\sum_{i=1}^m\mathsf{g}_{[(e,v)]}((\nabla_{u^{(1)}}u^{(1)})^{(i)},u^{(1)})=\tilde{\sg}((\nabla_{u^{(1)}}u^{(1)})^{(1)},u^{(1)}).$$
    We now write an evolution equation for $\tilde{\sg}(u^{(1)},u^{(1)})$. Observe that
    \begin{align*}
        \left|\frac{d}{dt}\tilde{\sg}(u^{(1)},u^{(1)})\right|&=2\left|\tilde{\sg}\left(\frac{du^{(1)}}{dt},u^{(1)}\right)\right|\\
        &\le2\left|\tilde{\sg}\left((\nabla^r_{u^{(0)}} u^{(0)})^{(1)},u^{(1)}\right)\right|+2\left|\tilde{\sg}\left((\nabla^r_{u^{(0)}} u^{(1)})^{(1)}+(\nabla^r_{u^{(1)}} u^{(0)})^{(1)},u^{(1)}\right)\right|.
    \end{align*}
    Because $u^{(0)}$ is bounded in $C^0$ and lies in a finite-dimensional space of smooth vector fields, we have that $u^{(0)}$ is bounded in $C^k$ for all $k\in\N$. Similarly,  $u^{(1)}$ also lies in a finite-dimensional space of smooth functions, so for any $k\in \mathbb{N}$, the $C^k$ norm of $u^{(1)}$ is controlled by the $C^0$ norm. Bounds on $\|u^{(1)}\|_{\tilde{\sg}}$ on $[0,T)$ then follow from the Cauchy-Schwarz and Gr\"onwall inequalities. 
\end{proof}

We now bound the remaining terms in the Taylor expansion of $u$ and subsequently use this to obtain uniform $C^1$ bounds of $u$ in a neighbourhood of the nonprincipal orbits. To do this, we write an evolution equation for $u^{(k)}$, $k\geq 2$ and observe that there are no terms that are nonlinear in $u^{(k)}$ (or contain $u^{(\ell)}$ for $\ell>k$). The bounds on $u^{(k)}$ imply that $u$ is bounded in $C^1$ in a possibly smaller neighbourhood of the nonprincipal orbit. We also prove $C^1$ bounds in a neighbourhood of each principal orbit by choosing a local frame of invariant vector fields and using a similar but simpler argument to the nonprincipal orbit case. Combined with the compactness of $M$ and the blow-up criterion of Theorem \ref{Eulerwellposed}, this proves Theorem \ref{maintheorem} in the case $M/G=[0,1]$.
\begin{prop}\label{noblow-upInterval}
    Suppose $M/G=[0,1]$. For any smooth divergence-free $G$-invariant vector field $u_0\in\fX^G(M)$, there exists a unique smooth solution $u:\mathbb{R}\to\fX^G(M)$ to \eqref{EulerCH1}--\eqref{Div0} such that $u(0,x)=u_0(x)$ for all $x\in M$.
\end{prop}
\begin{proof}
    Fix an arbitrary $T>0$. It suffices to prove that $\|u(t)\|_{C^1(M,\mathsf{g})}<\infty$ for $t\in[0,T)$ by the blow-up criterion in Theorem \ref{Eulerwellposed}. In fact, by compactness of $M$, it suffices to show that each $q\in M$ is contained in an open neighbourhood $U\subseteq M$ on which $\sup_{t\in [0,T)}\|u(t)\|_{C^1(U,\mathsf{g})}<\infty$. 

  First, assume that $q$ lies in a nonprincipal orbit, and without loss of generality, that $q\in G \gamma(0)$. Set $U_0$ to be the neighbourhood on which Lemma \ref{EWSmoothness} applies. In Lemma \ref{u1bounds}, we showed that $u^{(0)}$ and $u^{(1)}$ are uniformly bounded on $[0,T)$; we now demonstrate that $u^{(k)}$ is also uniformly bounded for all $k\ge 2$. We have the evolution equation
    $$\frac{d u^{(k)}}{dt}=\sum_{i,j=0}^{\infty}(-\nabla_{u^{(i)}} u^{(j)})^{(k)}=\sum_{\substack{i+j\le k+1 \\ i>0, j>0}}(-\nabla_{u^{(i)}} u^{(j)})^{(k)}$$
    by Lemma \ref{ivg}. For $k\ge 2$, we therefore have 
    \begin{align*}
        \frac{d u^{(k)}}{dt}=\sum_{i,j=1}^{k-1}\left(-\nabla_{u^{(i)}} u^{(j)}\right)^{(k)}-\left(\nabla_{u^{(1)}}u^{(k)}\right)^{(k)}-\left(\nabla_{u^{(k)}}u^{(1)}\right)^{(k)}.
    \end{align*}
We proceed by induction to prove that $u^{(k)}$ is uniformly bounded. Indeed, the inductive hypothesis gives us uniform control of the first term, and since $u^{(k)}$ appears linearly on the other terms, we can control the growth on $[0,T)$ with Gr\"onwall's inequality. 
    
On $U_0$, we can write
    $$u(t,x)=\sum_{i=1}^n u_i(t,r^2) W_i(x)$$
    as in Lemma \ref{EWSmoothness}. The bounds on $u^{(k)}$ for all $k\ge 0$ imply that $u_i(t,0)$ is bounded on $t\in[0,T)$ for all $i=1,\dots,n$. Using Corollary \ref{Christoffel}, we have 
   \begin{align*}
       \frac{d u_k}{dt}=\sum_{i,j=1}^n c_{ij}^{k} u_i u_j. 
   \end{align*}
   Writing $u_k(r^2)=u_k(0)+r^2 \tilde{u}_k(r^2)$, we find that the $\tilde{u}_k$ functions satisfy the equation
   \begin{align}\label{usingpet}
       \frac{d \tilde{u}_k}{dt}= r^2  \sum_{i,j}c^{k}_{ij}\tilde{u}_i\tilde{u}_j+ \sum_{i,j}c^{k}_{ij}\tilde{u}_iu_j(0)+ \sum_{i,j}c^{k}_{ij}u_i(0)\tilde{u}_j.
   \end{align}
   Since $T>0$ and the initial data is fixed, it is possible to restrict $U_0$ (thus restricting the size of $r^2$) to ensure boundedness of $\|\tilde{u}_i\|_{C^0(U_0)}$ all the way up until time $T$. Bounds on $\|\tilde{u}_i\|_{C^1(U_0)}$ then follow by differentiating \eqref{usingpet} in $r$, from which we obtain $C^1$ bounds on $u$ in $U_0$.  

  Let us now assume that $q$ is on the principal orbit $G \gamma(r_0)$. In this case, Lemma \ref{vectorlem} gives $u=\sum_{i=1}^{M} v_i(r) Y_i$ for some smooth and scalar functions $v_i$. Since $\nabla_{Y_i}Y_j$ is another smooth and $G$-invariant vector field on $M_p$, we can write 
  \begin{align*}
      -\nabla_{Y_i}Y_j=\sum_{k} d_{ij}^{k}(r) Y_k
  \end{align*}for some smooth and scalar functions $d_{ij}^{k}(r)$. We can therefore write the evolution equation for $u$ as 
       \begin{align*}
       \frac{d v_k}{dt}=\sum_{i,j} d_{ij}^{k} v_i v_j. 
   \end{align*}
   Writing $v_k(r)=v_k(r_0)+(r-r_0)\tilde{v}_k(r)$, we obtain the equations 
\begin{align*}
       \frac{d \tilde{v}_k}{dt}= (r-r_0)  \sum_{i,j}d^{k}_{ij}\tilde{v}_i\tilde{v}_j+ \sum_{i,j}d^{k}_{ij}\tilde{v}_iv_j(r_0)+ \sum_{i,j}d^{k}_{ij}v_i(r_0)\tilde{v}_j,
   \end{align*}
and since $v_i(r_0)$ is bounded, and we can argue similarly to the nonprincipal orbit case. 
\end{proof}

The case $M/G=S^1$ is more interesting dynamically, since the space of horizontal divergence-free $G$-invariant vector fields is nontrivial by Lemma \ref{BCs}. Because of this, the Euler equations do not decouple orbit-wise. Nevertheless, the fact that the space of horizontal divergence-free vector fields is finite-dimensional can be combined with conservation of energy to conclude global regularity for the horizontal component of the solution, from which a maximum principle argument can be used to conclude global regularity for the entire solution. Since we have already proven global regularity for the cases $M/G=[0,1]$, proving global regularity for $M/G=S^1$ yields Theorem \ref{maintheorem}.

\begin{proof}[Proof of Theorem \ref{maintheorem}]
    Due to Proposition \ref{noblow-upInterval}, the only case we need to consider is $M/G=S^1$. Once again, it suffices to fix an arbitrary $T>0$, and show that if $u\in\fX^G(M)$ satisfies the initial condition $u(0)=u_0$, and also solves \eqref{EulerCH1}--\eqref{Div0} on $[0,T)$, then $\sup_{t\in [0,T)}\|u(t)\|_{C^1}<\infty$. By Lemma \ref{BCs}, the horizontal component $h$ is given by
    $$h(t,r)=c(t)h_0(r),$$
    with $h_0$ being bounded and everywhere nonzero. We will show that $c$ is uniformly bounded using the standard conservation of $L^2$ energy. Observe that
    $$\sg(u,u)=h^2\sg(\partial_r,\partial_r)+\sg(v,v)=c^2h_0^2+\sg(v,v).$$
    The $L^2$ energy $E$ therefore satisfies
    $$2E=\int_M \sg(u,u)=c^2\underbrace{\int_M h_0^2}_{>0}+\underbrace{\int_M\sg(v,v)}_{\ge0}.$$
    Denoting the initial energy by $E_0$, we see that
    $$c^2\le\frac{2E_0}{\int_M h_0^2}.$$
    Thus, $c$ is uniformly bounded. It follows that $\sup_{t\in [0,T)} \|h(t)\|_{C^{k}}<\infty$ for any $k\in \mathbb{N}$. 

    To find a bound on $v$, we use the second equation of \eqref{EulerCH1}. Observe that
    \begin{align*}
        \frac{d}{dt} \mathsf{g}(v,v)= -h \partial_r \mathsf{g}(v,v)-2h \mathsf{g}_r (S_r v,v). 
    \end{align*}
  The first term is a transport term and the second term is controlled by a quantity which is linear in $\sg(v,v)$. The maximum principle combined with Gr\"onwall's inequality gives the $C^0$ estimate $\sup_{t\in [0,T)} \mathsf{g}(v,v)<\infty$.

  For the $C^1$ bounds, observe that since all of our orbits are principal, and the orbit space is $S^1$, Lemma \ref{vectorlem} gives us  $n_0=\dim(\mathfrak{m}_0)$ $G$-invariant vector fields $Y_1,\dots,Y_n$ that are defined globally and are everywhere linearly independent. These vector fields have the property that any smooth $G$-invariant vector field $X$ on $M$ can be written as $Z=\sum_{i=1}^{n} z_i(r) Y_i$, where $Z_i:S^1\to \mathbb{R}$ are smooth. As a consequence, we have $-\nabla_{Y_i}Y_j= \sum_{k=1}^{n}d_{ij}^{k} Y_k$ for some smooth functions $d_{ij}^{k}:S^1\to \mathbb{R}$. Writing now $v=\sum v_i Y_i$, with $v_i(t,r)$, $r\in S^1$, we have 
  \begin{align*}
      \frac{dv_k}{dt}=-h\partial_r (v_k)+\sum_{i=1}^{n}\ell_iv_i+\sum_{i,j=1}^{n}q_{ij}v_iv_j, 
  \end{align*}
  where $\ell_i,q_{ij}$ are smooth functions of $r$. 
  We thus have 
  \begin{align*}
      \frac{d}{dt}\sum_{k=1}^{n}(\partial_r v_k)^2=-h\partial_r \left(\sum_{k=1}^{n}(\partial_r v_k)^2\right)+F(r,v_1,\dots,v_n,\partial_r v_1,\dots,\partial_r v_n), 
  \end{align*}
  where $F$ is smooth and quadratic in the $\partial_r v_i$ entries. The scalar maximum principle again gives uniform $C^1$ bounds on the $v_i$ functions on $[0,T)$, from which follow the required $C^1$ bounds on $v$ and $u$. 
\end{proof}

   \section{Cohomogeneity one manifolds with boundary}
The global existence result of Theorem \ref{maintheorem} concerns cohomogeneity one manifolds \textit{without boundary}, but we can see that the same result holds if the cohomogeneity one Riemannian manifold $(M,\mathsf{g})$ includes boundary. In fact, if the boundary is nonempty, the boundary conditions force our vector field to be vertical.
 \begin{lem}
     If the compact, connected and smooth cohomogeneity one Riemannian manifold $(M,\mathsf{g})$ has nonempty boundary, then any smooth, divergence-free and $G$-invariant vector field $u$ which satisfies \eqref{BCs0} is vertical. 
 \end{lem}
 \begin{proof}
     By definition, $M/G$ is compact and one-dimensional, and we also find that the projection $\pi: M\to M/G$ is a Riemannian submersion. Thus, $\pi (\partial M)\subseteq \partial (M/G)$, which implies that $M/G$ must be diffeomorphic to $[0,1]$. Without loss of generality, suppose that $x\in \partial M$ and $\pi(x)=0$, which implies that $\pi^{-1}(0)\subseteq \partial M$. Let $\gamma:[0,R]\to M$ be a maximal geodesic with $\gamma(0)=x$, and $\gamma'(0)$ perpendicular to $G x$. By \cite[Proposition 3.78]{AlexandrinoBettiol15}, this geodesic intersects each orbit it meets orthogonally. Furthermore, the geodesic intersects all orbits at least once. Therefore, it suffices to prove that $u$ is vertical along the geodesic $\gamma$. Like in the empty boundary case, the divergence free condition gives the ODE  \eqref{Div0} for the horizontal component of the vector field. The boundary condition \eqref{BCs0} then implies that $h(0)=0$, so that $h$ must vanish uniformly.
 \end{proof}
Thus, the equations decouple again. The proof then proceeds in an almost identical fashion to that of Theorem \ref{maintheorem}.
\begin{thm}
   Suppose the compact Lie group $G$ acts with cohomogeneity one and isometrically on the compact Riemannian manifold $(M,\mathsf{g})$ whose boundary includes a principal orbit. For any smooth, divergence-free and $G$-invariant vector field $u_0\in \fX(M)$, there exists a smooth $u:\mathbb{R}\to \fX(M)$ and $p:\mathbb{R}\to C^{\infty}(M)$ satisfying \eqref{Euler}, \eqref{BCs0} and the initial condition $u(0)=u_0$. 
\end{thm}
\begin{proof}
As in the proof of Proposition \ref{noblow-upInterval}, it suffices to choose an arbitrary $T>0$, and show that if $u: [0,T)\to \fX(M)$ and $p:[0,T)\to C^{\infty}(\mathbb{R})$ is a solution, then each $x\in M$ is contained in an open neighbourhood $U$ on which $\sup_{t\in [0,T)} \| u(t)\|_{C^1(U,\mathsf{g})}<\infty$. In fact, given the work of Proposition \ref{noblow-upInterval}, it suffices to do this in the case $x\in \partial M$, but the argument proceeds identically. Indeed, since $x$ is on a principal orbit, we can find a finite number of smooth and $G$-invariant vector fields $\{Y_i\}_{i=1}^{M}$ defined locally, so that $u=\sum_{i=1}^{M}v_i(t,r)Y_i$  for some smooth and scalar functions of $r$. It is possible to show that the scalar functions $v_i$ are bounded in $C^1$, independently of $t$, and the required $C^1$ estimates on $u$ follow. 
\end{proof}

\clearpage

\end{document}